\documentclass[reqno,12pt]{amsart}
\def\vint_#1{\mathchoice%
          {\mathop{\kern 0.2em\vrule width 0.6em height 0.69678ex depth -0.58065ex
                  \kern -0.8em \intop}\nolimits_{\kern -0.4em#1}}%
          {\mathop{\kern 0.1em\vrule width 0.5em height 0.69678ex depth -0.60387ex
                  \kern -0.6em \intop}\nolimits_{#1}}%
          {\mathop{\kern 0.1em\vrule width 0.5em height 0.69678ex
              depth -0.60387ex
                  \kern -0.6em \intop}\nolimits_{#1}}%
          {\mathop{\kern 0.1em\vrule width 0.5em height 0.69678ex depth -0.60387ex
                  \kern -0.6em \intop}\nolimits_{#1}}}
\def\vintslides_#1{\mathchoice%
          {\mathop{\kern 0.1em\vrule width 0.5em height 0.697ex depth -0.581ex
                  \kern -0.6em \intop}\nolimits_{\kern -0.4em#1}}%
          {\mathop{\kern 0.1em\vrule width 0.3em height 0.697ex depth -0.604ex
                  \kern -0.4em \intop}\nolimits_{#1}}%
          {\mathop{\kern 0.1em\vrule width 0.3em height 0.697ex de pth -0.604ex
                  \kern -0.4em \intop}\nolimits_{#1}}%
          {\mathop{\kern 0.1em\vrule width 0.3em height 0.697ex depth -0.604ex
                  \kern -0.4em \intop}\nolimits_{#1}}}

\usepackage{a4wide}
\usepackage{mathtools}
\usepackage{amsthm}
\usepackage{amssymb}
\usepackage{hyperref}
\usepackage{color}
\mathtoolsset{showonlyrefs}
\usepackage[obeyFinal]{todonotes}
\usepackage{csquotes}
\usepackage{graphicx}
\usepackage{epsfig,color}
\usepackage{enumitem}

\numberwithin{equation}{section}
\newtheorem{theorem}{Theorem}[section]
\newtheorem{lemma}[theorem]{Lemma}
\newtheorem{corollary}[theorem]{Corollary}
\theoremstyle{definition}
\newtheorem{definition}[theorem]{Definition}
\theoremstyle{remark}
\newtheorem{rem}[theorem]{Remark}

\renewcommand{\d}{{\mathrm d}}
\newcommand{\R}{\mathbb{R}}
\renewcommand{\S}{\mathbb{S}}
\newcommand{\N}{\mathbb{N}}
\newcommand{\Q}{\mathcal{Q}}
\newcommand{\F}{\mathcal{F}}
\newcommand{\D}{\mathcal{D}}
\newcommand{\Int}{\mathrm{int\,}}
\newcommand{\diam}{\mathrm{diam}}
\newcommand{\dist}{\mathrm{dist}}
\newcommand{\spt}{\mathrm{spt}}

\begin{document}

\title[A density result for homogeneous Sobolev spaces on planar domains]
{A density result for homogeneous Sobolev spaces\\ on planar domains}

\author{Debanjan Nandi}
\author{Tapio Rajala}
\author{Timo Schultz}

\address{Department of Mathematics and Statistics \\
         P.O. Box 35 (MaD) \\
         FI-40014 University of Jyv\"as\-kyl\"a \\
         Finland}
\email{debanjan.s.nandi@jyu.fi}
\email{tapio.m.rajala@jyu.fi}
\email{timo.m.schultz@student.jyu.fi}

\thanks{All authors partially supported by the Academy of Finland. 
}
\date{\today}


\begin{abstract}
We show that in a bounded simply connected  planar domain $\Omega$ the smooth Sobolev functions $W^{k,\infty}(\Omega)\cap C^\infty(\Omega)$ are dense in the homogeneous Sobolev spaces $L^{k,p}(\Omega)$.
\end{abstract}

\maketitle


\section{Introduction}

By the result of Meyers-Serrin \cite{MS} it is known that $C^\infty(\Omega)$ is dense in $W^{k,p}(\Omega)$ for every open set $\Omega$ in $\R^d$.
The space $C^\infty(\R^d)$ is not always dense in $W^{k,p}(\Omega)$, for example when $\Omega$ is a slit disk.
However, a slit disk is not a very appealing example as it is not the interior of its closure. Counterexamples for the density satisfying $\Omega = \text{int}(\overline{\Omega})$ were given by Amick \cite{Amick} and Kolsrud \cite{Kolsrud}. In fact, in these examples even $C(\overline{\Omega})$ is not dense in $W^{k,p}(\Omega)$. Going further in counterexamples, 
O'Farrell \cite{OFarrell} constructed a domain satisfying $\Omega = \text{int}(\overline{\Omega})$  where $W^{k,\infty}(\Omega)$ is not dense in $W^{k,p}(\Omega)$ for any $k$ and $p$. The domain constructed by O'Farrell was infinitely connected. From the recent results of Koskela-Zhang \cite{KZ} and Koskela-Rajala-Zhang \cite{KTZ} we can conclude that this is necessary for such constructions in the plane, since $W^{1,\infty}(\Omega)$ is dense in $W^{1,p}(\Omega)$ for all finitely connected bounded planar domains (see also the earlier work by Giacomini-Trebeschi \cite{GT}). Further examples of domains where $W^{1,p}(\Omega)$ is not dense in $W^{1,q}(\Omega)$ were constructed by Koskela \cite{Koskela} and Koskela-Rajala-Zhang \cite{KTZ}.

In this note we continue the study of density of $W^{k,\infty}(\Omega)$ in $W^{k,p}(\Omega)$.
Let us remark that such density clearly holds in the case where the Sobolev functions in $W^{k,p}(\Omega)$ can be extended to Sobolev functions defined on the whole $\R^2$. By work of Jones \cite{J}, this is true when $\partial\Omega$ is a quasi-circle. (See also the works \cite{Gold1,Gold2,Gold3}.)
Geometric characterizations of Sobolev extension domains are known, especially in the planar simply connected domains when $k=1$, see \cite{BK,Koskela2,Shvartsman,KRZ}.

Being an extension domain is only a sufficient condition for the density. For example, there are Jordan domains $\Omega$ and functions $f \in W^{1,p}(\Omega)$ that cannot be extended to a function in $W^{1,p}(\R^2)$. However, global smooth functions are dense in $W^{1,p}(\Omega)$ for any Jordan domain and any $p \in [1,\infty]$, see Lewis \cite{Lewis} and Koskela-Zhang \cite{KZ}. For $W^{k,p}(\Omega)$ with $k \ge 2$ this is still unknown.

In \cite{SSS} Smith-Stanoyevitch-Stegenga studied the density of $C^\infty(\R^2)$ as well as the density of functions in $C^\infty(\Omega)$ with bounded derivatives, in $W^{k,p}(\Omega)$. For the latter class they obtained a density result assuming $\Omega$ to be starshaped or to satisfy an interior segment condition. For the smaller class of functions $C^\infty(\R^2)$ they also required an extra assumption on the boundary points to be $m_2$-limit points. (See also Bishop \cite{Bishop} for a counterexample on a related question.)

The result of Koskela-Zhang \cite{KZ} showing that $W^{1,\infty}(\Omega)$ is dense in $W^{1,p}(\Omega)$ for every bounded simply connected planar domain was generalized to higher dimensions by Koskela-Rajala-Zhang \cite{KTZ}. They showed that simply connectedness is not sufficient to give such a density result, but Gromov hyperbolicity in the hyperbolic distance is. 
In this paper we provide another generalization to the Koskela-Zhang result by going to higher order Sobolev spaces. 
We show that if we restrict attention to the homogenous norm, then being simply connected is sufficient for domains in the plane.


For a domain $\Omega\subset\mathbb{R}^2$ and $p\in [1,\infty)$, by homogenous Sobolev space $L^{k,p}(\Omega)$ we mean functions with $p$-integrable distributional derivatives of order $k$;
$$L^{k,p}(\Omega)=\{u\in L^1_{loc}(\Omega):\nabla^{\alpha}u\in L^{p}(\Omega),\;\text{if}\;|\alpha|=k\},$$ with semi-norm $\sum_{|\alpha|= k}\|\nabla^{\alpha}u\|_{L^p(\Omega)}$,
where $\alpha$ is any $2$-vector of non-negative integers and  $|\alpha|$ is its $\ell_1$-norm. The (non-homogenous) Sobolev space $W^{k,p}(\Omega)$ is defined as 
$$W^{k,p}(\Omega)=\{u\in L^1_{loc}(\Omega):\nabla^{\alpha}u\in L^{p}(\Omega),\;\text{if}\;|\alpha|\leq k\},$$ with norm $\sum_{|\alpha|\leq k}\|\nabla^{\alpha}u\|_{L^p(\Omega)}$.
\begin{theorem}\label{main_thm}
Let $k\in\mathbb{N}$, $p \in [1,\infty)$ and $\Omega\subset\mathbb{R}^2$ be a bounded simply connected domain. Then the subspace $W^{k,\infty}(\Omega)\cap C^{\infty}(\Omega)$ is dense in the space $L^{k,p}(\Omega)$. 
\end{theorem}

The approach in \cite{KTZ} differs from ours in that there the approximating functions are defined via shifting matters to the disk via the Riemann mapping. Instead, we directly make a Whitney decomposition of the domain and a rough reflection to define our approximating sequence. We achieve this via an elementary use of simply connectedness in the plane. In both of these approaches the values of the function in a suitable compact set are used to define a smooth function in the entire domain which approximates the original function in Sobolev norm. For this we employ similar tools as used by Jones in \cite{J}. 


In $p$-Poincar\'e domains, that is domains $\Omega$ where a $p$-Poincar\'e inequality
\[
\int_\Omega|u - u_D|^p\,\d x \le C \int_\Omega |\nabla u|^p\,\d x
\]
holds, we can bound the integrals of the lower order derivatives by the integrals of the higher order ones and thus we obtain the following corollary to our Theorem \ref{main_thm}.

\begin{corollary}\label{cor}
Let $k \in \N$, $p \in [1,\infty)$ and $\Omega \subset \R^2$ be a bounded simply connected $p$-Poincar\'e domain.
Then $W^{k,\infty}(\Omega)\cap C^{\infty}(\Omega)$ is dense in the space $W^{k,p}(\Omega)$.
\end{corollary}

For instance H\"older-domains are $p$-Poincar\'e domains for $p \ge 2$, see Smith-Stegenga \cite{SmithStegenga}. It still remains an open question whether Corollary \ref{cor} holds if one drops the assumption of being a $p$-Poincar\'e domain.

Next we come to the question of density of $C^{\infty}(\R^2)$ functions in $L^{k,p}(\Omega)$ in our setting of bounded simply connected domains. We have the following corollary which is analogous to \cite[Corollary 1.2]{KZ}, where it is shown that $\Omega$ being Jordan is sufficient. A small modification of the argument there applies to our situation as well. See the end of Section \ref{sec:approx} for the proof.

\begin{corollary}\label{cor_global_density}
Let $k\in\N$, $p\in [1,\infty)$ and $\Omega\subset\mathbb{R}^2$ be a Jordan domain. 
Then $C^{\infty}(\R^2)$ is dense in the space $L^{k,p}(\Omega)$.
\end{corollary}

In Section \ref{sec:preli}, we collect the necessary ingredients which will be used for defining the approximating sequence; these include a suitable Whitney-type decomposition of a simply connected domain and a local polynomial approximation of Sobolev functions. In Section \ref{sec:decomp} we describe a partition of the domain using the Whitney-type decomposition of Section \ref{sec:preli}, which is needed for obtaining a suitable partition of unity. Then in Section \ref{sec:approx}, we define the approximating sequence and present the necessary estimates for proving Theorem \ref{main_thm} and Corollary \ref{cor_global_density}.

\section{Preliminaries}\label{sec:preli}

For sets $A,B \subset \R^2$ we denote the diameter of $A$ by $\diam(A)$ and the distance between $A$ and $B$ by $\dist(A,B)$. 
We denote by $B(x,r)$ the open ball with center $x \in \R^2$ and radius $r>0$ and more generally, by $B(A,r)$ the open $r$-neighbourhood of a set $A \subset \R^2$.
Given a connected set $E\subset \R^2$ and points $x,y\in E$, we define the inner distance $\d_{E}(x,y)$ between $x$ and $y$ in $E$ to be the infimum of lengths of curves in $E$ joining $x$ to $y$. (Notice that in general the infimum might have value $\infty$.)
We write the inner distance in $E$ between sets $A, B \subset E$ as $\dist_E(A,B)$.
 
With a slight abuse of notation, by a curve $\gamma$ we refer to both, a continuous mapping $\gamma \colon [0,1] \to \R^2$ and its image $\gamma([0,1])$. Given two curves $\gamma_1,\gamma_2\colon [0,1]\to\mathbb{R}^2$ such that $\gamma_1(1)=\gamma_2(0)$, we denote by $\gamma_1\ast\gamma_2 \colon [0,1]\to\mathbb{R}^2$ the concatenated curve $\gamma_{1}\ast\gamma_2(t)=\gamma_1(2t)$ for $t\leq 1/2$ and $\gamma_1\ast\gamma_2(t)=2t-1$ for $t\geq 1/2$. We denote the length of a curve $\gamma$ by $L(\gamma)$.

We will use the following facts in plane topology whose proofs can be found in the book of Newman \cite[Chapter VI, Theorem 5.1 and Chapter V, Theorem 11.8]{Newman}.

\begin{lemma}\label{lma:domaincut}
 Let $\Omega$ be a simply connected domain in $\R^2$ and $\gamma \colon [0,1]\to \R^2$ a continuous curve that is injective on $(0,1)$, whose endpoints $\gamma(0)$ and $\gamma(1)$ are in $\partial \Omega$ and interior $\gamma((0,1))$ in $\Omega$. Then $\Omega \setminus \gamma$ has two connected components, both of which are simply connected.
 
 In the case where $\Omega$ is Jordan and $\gamma$ is homeomorphic to a closed interval, the two connected components of $\Omega \setminus \gamma$ have boundaries $\gamma \cup J_1$ and $\gamma \cup J_2$, where $J_1$ and $J_2$ are the two connected components of $\partial\Omega \setminus \gamma$. 
\end{lemma}

\subsection{A dyadic decomposition}\label{dyadic_decomposition}

Although it is standard to consider a Whitney decomposition of a domain in $\R^d$ (see for instance Whitney \cite{whitney} or the book of Stein \cite[Chapter VI]{stein}), we will use a precise construction of such a decomposition. We present this construction below. Here and later on we denote the sidelength of a square $Q$ by $l(Q)$.

For notational convenience we start the Whitney decomposition below from squares with sidelength $2^{-1}$. Formally, by rescaling, we may consider all bounded domains $\Omega \subset \R^2$ to have $\diam(\Omega) \le 1$ in which case no Whitney decomposition would have squares larger than the ones used below regardless of the starting scale.

\begin{definition}[Whitney decomposition]\label{def:whitney}
Let $\Omega\subset \R^2$ be a bounded (simply connected) open set. Let $\Q_n$ be the collection of all closed dyadic squares of sidelength $2^{-n}$. 
Define a Whitney decomposition as $\tilde\F\coloneqq\bigcup_{n\in\N}\tilde\F_n$ where the sets $\tilde\F_n$ are defined recursively as follows. Define 
\begin{align} \tilde\F_1\coloneqq \left\{ Q\in \Q_1: \bigcup_{\substack{Q'\in \Q_1 \\ Q'\cap Q\neq \emptyset }}Q'\subset \Omega\right\}
\end{align}
and 
\begin{align}
\tilde\F_{n+1}\coloneqq \left\{ Q\in \Q_{n+1}:Q\not\subset \tilde{F}_n \textrm{ and }\bigcup_{\substack{Q'\in \Q_{n+1} \\ Q'\cap Q\neq \emptyset }}Q'\subset \Omega\right\},
\end{align}
where $\tilde{F_n}=\bigcup_{j\le n}\bigcup_{Q\in\tilde\F_j}Q$.
\end{definition}


\begin{lemma}\label{lma:whitneyprop}
A Whitney decomposition given by Definition \ref{def:whitney} has the following properties.
\begin{enumerate}[label=\textbf({W\arabic*)},ref=(W\arabic*)]
\item $\Omega=\bigcup_{Q\in \tilde\F}Q $ \label{eka}
\item $l(Q)<\dist(Q,\Omega^c)\le3\sqrt2 l(Q)=3\diam(Q)$ for all $Q\in\tilde\F$\label{toka}
\item $\Int Q_1\cap \Int Q_2= \emptyset$ for all $Q_1,Q_2\in \tilde\F, Q_1\ne Q_2$\label{kolmas}
\item If $Q_1,Q_2\in \tilde\F$ and $Q_1\cap Q_2\neq\emptyset$, then $\frac{l(Q_1)}{l(Q_2)}\le 2$.\label{neljas}
\end{enumerate}
\end{lemma}
\begin{proof}
Although the proof is very elementary, we give it here for completeness.

For \ref{eka}, take any $x \in \Omega$ and $n \in \N$ such that $x \in Q \in \Q_n$, where $2^{-n+2}\sqrt{2} < \dist(x,\Omega^c) \le 2^{-n+3}\sqrt{2}$. Then for any $Q' \in \Q_n$ with $Q' \cap Q \ne \emptyset$ we have $Q' \subset \Omega$. Hence by definition either $Q \in \tilde F_n$ or $x \in Q \subset Q'' \in \tilde \F_i$ for some $i < n$.

In order to see \ref{toka}, let $Q \in \tilde \F_n$. Then all $Q' \subset \Omega$ for all $Q' \in \Q_n$ with $Q'\cap Q \ne \emptyset$. Consequently, $\dist(Q,\Omega^c) > 2^{-n} = l(Q)$.
For the upper bound, suppose $\dist{Q,\Omega^c}> 3\sqrt{2}2^{-n}$. Let $Q_2 \in \Q_{n-1}$ be such that $Q \subset Q_2$. Then $\dist(Q_2,\Omega^c) > \sqrt{2}2^{-n+1}$ and so $Q_3 \subset \Omega$ for all $Q_3 \in \Q_{n-1}$ for which $Q_2\cap Q_3 \ne \emptyset$. Thus $Q_2 \in \tilde F_{n-1}$ or $Q_2 \subset Q_4 \in \tilde F_i$ for some $i < n-1$. In either case, $Q \notin \tilde \F_n$ giving a contradiction.

Property \ref{kolmas} holds by the recursion in the definition and the fact that the dyadic squares are nested.

Suppose \ref{neljas} is not true. Then there exist $Q_1 \in \tilde\F_n$ and $Q_2 \in \tilde \F_m$ with $n < m - 1$ and $Q_1\cap Q_2\neq\emptyset$. Let $Q_3 \in \tilde F_{n+1}$ be such that $Q_2 \subset Q_3$. Then
\[
\bigcup_{\substack{Q'\in \Q_{n+1} \\ Q'\cap Q_3\neq \emptyset }}Q'\subset 
\bigcup_{\substack{Q'\in \Q_{n} \\ Q'\cap Q_1\neq \emptyset }}Q'\subset \Omega
\]
and so either $Q_3 \in \tilde \F_{n+1}$ or $Q_3 \subset \tilde F_{n}$. In both cases $Q_2 \subset \tilde F_{n+1}$ and so $Q_2 \notin \tilde \F_m$.
\end{proof}

By a chain of dyadic squares $\{Q_i\}_{i=1}^m$ we mean a collection of sets $Q_i \in \tilde{\F}$ such that 
$Q_i \cap Q_{i+1}$ is a non-degenerate line segment for all $i \in \{1,\dots,m-1\}$. We say that the chain connects $Q_1$ and $Q_m$.

\subsection{Approximating polynomials}\label{sec:approxpoly}

We record here the following two Lemmas from \cite{J} which will be used when estimating the approximation in Section \ref{sec:approx}. 
\begin{lemma}[Lemma 2.1, \cite{J}] \label{norm_equivalence}
Let $Q$ be any square in $\mathbb{R}^2$ and $P$ be a polynomial of degree $k$ defined in $\mathbb{R}^2$. Let $E,F\subset Q$ be such that $|E|,|F|>\eta |Q|$ where $\eta>0$. Then
$$\|P\|_{L^p(E)}\leq C(\eta, k)\|P\|_{L^p(F)}.$$
\end{lemma}

Given a function $u\in C^{\infty}(\Omega)$ and a bounded set $E\subset\Omega$, we define (see \cite{J}) the polynomial approximation of $u$ in $E$ , $P_k(u,E)$ to be the polynomial of order $k-1$ which satisfies
$$\int_E \nabla^{\alpha}(u-P_k(u,E))=0$$ for each $\alpha=(\alpha_1,\alpha_2)$ such that $|\alpha|=\alpha_1+\alpha_2\leq k-1$. Once $k$ is fixed, we denote the polynomial approximation of $u$ in a dyadic square $Q$ as $P_Q$

The next lemma is a consequence of Poincar\'e inequality for Lipschitz domains.

\begin{lemma}[Lemma 3.1, \cite{J}]\label{chaining}
 Let $\Omega \subset \R^2$ be a bounded simply connected domain
 and $\tilde \F$ a Whitney decomposition of $\Omega$. 
 Fix $\alpha$ such that $|\alpha|\leq k$.
 Let 
 $\{Q_i\}_{i=1}^m$ in $\tilde\F$ be a chain of dyadic squares in $\tilde{\mathcal{F}}$.
  Then we have 
 $$
 \|\nabla^{\alpha}(P_{Q_1}-P_{Q_m})\|_{L^p(Q_1)}\leq Cl(Q_1)^{k-|\alpha|}\|\nabla^ku\|_{L^p(\bigcup_{i=1}^m Q_i)},
 $$
 where $\nabla^k u$ is the vector $(\nabla^{\alpha}u)_{|\alpha|=k}$ normed by the $\ell_2$-norm and $C=C(m)$.
\end{lemma}
In what follows, given $\beta=(\beta_1,\beta_2)$ and $\alpha=(\alpha_1,\alpha_2)$, we write $\beta\leq \alpha$ if the inequality holds coordinate-wise.

\section{Decomposition of the domain}\label{sec:decomp}

\begin{figure}
\includegraphics[width=0.7\textwidth]{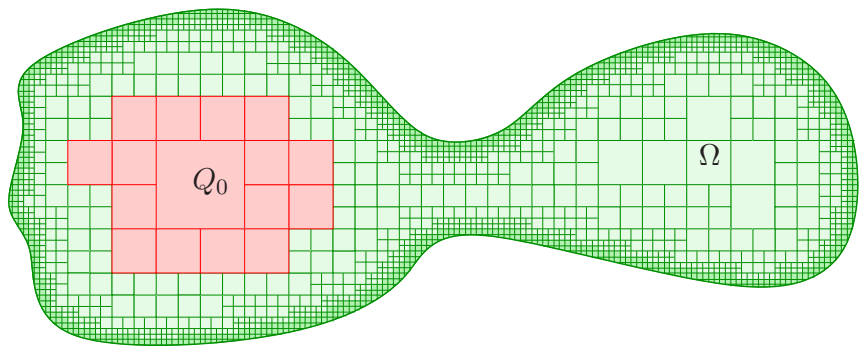}
     \caption{A core part $D_n$ is selected from the Whitney decomposition of $\Omega$ by taking the connected component containing $Q_0$ of the interior of the union of Whitney squares with sidelength at least $2^{-n}$.}
   \label{fig:connected}
 \end{figure}
 
From now on we fix a bounded simply connected domain $\Omega \subset \R^2$ and a
 Whitney decomposition $\tilde \F$ of $\Omega$ given by Definition \ref{def:whitney}. 
For our purposes we need to choose at each level a nice enough subcollection of $\tilde\F_n$, namely we take connected components of the Whitney decomposition (see Figure \ref{fig:connected}).  More precisely we fix $Q_0\in\tilde\F_1$ and for each $n \in \N$ let $C_n$ be the connected component of the interior of $\tilde{F_n}$ that has $\Int Q_0$ as a subset. We define 
\[
\F_{n,j}\coloneqq \{Q\in \tilde\F_j: \Int Q\subset C_n\}
\]
and using this the families of squares
\[
  \F_n\coloneqq \F_{n,n} ,\qquad  \D_n\coloneqq \bigcup_{j\le n}\F_{n,j} 
\]
and the corresponding sets for two of the above collections by
\[
F_n\coloneqq \bigcup_{Q\in\F_n}Q \quad\text{ and }\quad D_n\coloneqq \bigcup_{Q\in\D_n}Q=\overline{C}_n.
\]
The collection of boundary layer squares in $\D_n$ is denoted by
\[
\partial \D_n\coloneqq \left\{Q\in \D_n: Q\cap \overline{(\Omega\setminus D_n)}\neq\emptyset\right\}.
\]

With this notation we have the following lemma.

\begin{lemma}\label{properties}The above collections have the properties:
\begin{enumerate}[label={(\roman*)},ref=(\roman*)]
\item $\D_n\subset \D_{n+1}$ for all $k\in\N$.\label{yks}
\item $\Omega=\bigcup_{n\in \N}D_n $. \label{kaks}
\item If $Q_1,Q_2\in \F_n$ and $Q_1\cap Q_2$ is a singleton, then there exists $Q_3\in\D_n$ for which $Q_1\cap Q_2\cap Q_3\neq \emptyset$. \label{kol}
\item If $Q\in \partial \D_n$, then $Q\in \F_n$. \label{nel}
\item If $Q\in \partial \D_n$, then $Q\cap \overline{(\Omega\setminus \tilde F_n)}\neq\emptyset$.\label{viis}
\item The set $C_n$ is simply connected.\label{kuus}
\end{enumerate}
\end{lemma} 

\begin{proof}
The property \ref{yks} is obvious by the definitions of $\F_{n,j}$ and $\D_n$ since $C_n\subset C_{n+1}$. 

For \ref{kaks} it suffices to prove that for every $Q\in \tilde \F$ there exists $n\in\N$ so that $Q\in \D_n$. Let $Q\in \F_n$. Since $\Omega$ is connected and open, there exists a path $\gamma$ in $\Omega$ joining $Q$ to $Q_0$. By the fact that $\tilde F_j\subset \Int \tilde F_{j+1}$ and the property \ref{eka} of the decomposition $\tilde\F$ we have that $\Omega=\cup_{j\in\N}\Int\tilde F_j$. Then by the compactness of $\gamma$ there exists $m\ge n$ so that $\gamma\subset \Int \tilde F_m$. Hence $Q\in\D_m$.

\begin{figure}
\includegraphics[width=0.5\textwidth]{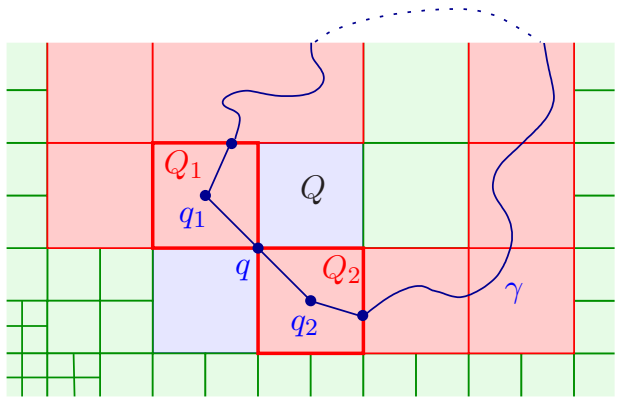}
     \caption{The constructed Jordan curve $\gamma$ in the proof of Lemma 3.1 $(\ref{kol})$ has in its interior domain a dyadic square $Q$ that also has to be an element of $\D_n$.}
   \label{fig:curves}
 \end{figure}
 
For \ref{kol} let $Q_1,Q_2\in \F_n$ be so that $Q_1\cap Q_2$ is a singleton $\{q\}$. Assume that the claim is false. Then for the two squares $Q\in\Q_n$ that intersect both $Q_1$ and $Q_2$ it is true that $Q\not\in \tilde\F_n$ and $Q\not\subset Q'$ for all $Q'\in\tilde\F_{n-1}$. Let $q_1$ and $q_2$ be the centres of the squares $Q_1$ and $Q_2$ respectively. Consider a curve $\gamma'\colon [0,1]\to \Omega$ for which $\gamma'_0=q_1$, $\gamma'_1=q_1$ and $\gamma'\subset C_n$. Such a curve exists by the definition of $C_n$. We may also assume that $\gamma'$ is an injective curve. Let $t_0\coloneqq \sup\{t: \gamma'(t)\in Q_1\}$ and $t_1\coloneqq \inf\{t\ge t_0: \gamma'(t)\in Q_2\}$. Define a Jordan curve 
\[
\gamma\coloneqq \gamma^1*\gamma^2*\gamma'\lvert_{[t_0,t_1]}*\gamma^3*\gamma^4,
\]
where $\gamma^1$, $\gamma^2$, $\gamma^3$ and $\gamma^4$ correspond to the line segments $[q, q_1]$, $[q_1,\gamma'_{t_0}]$, $[\gamma'_{t_1}, q_2]$ and $[q_2,q]$ respectively. By Jordan curve theorem $\gamma$ divides $\R^2$ into two components, one of which is precompact (see Figure \ref{fig:curves}). 
Denote the precompact component by $A$.

For small enough ball $B$ around $q$ we have by the definition of $\gamma$ that $B\setminus\gamma$ has exactly two components. Since $\gamma$ is a Jordan curve one of those components has to contain an interior point of $A$ and thus the whole component lies inside $A$. On the other hand that component has to intersect with one of the dyadic squares in $\Q_n$ touching both $Q_1$ and $Q_2$ (but being different from $Q_1$ and $Q_2$). Let $Q\in \Q_n$ be that square. Now for all the neighbouring squares $\tilde Q\in \Q_n$ (except the opposite one) of $Q$ either $\tilde Q\cap \gamma([0,1])\neq\emptyset$ implying that $\tilde Q\in \D_n$ or $\tilde Q$ is in the precompact component of $\R^2\setminus \gamma([0,1])$ and thus by simply connectedness $\tilde Q\subset \Omega$. Since $Q_1\in \F_n$, also the opposite square of $Q$ is a subset of $\Omega$. Hence $Q\in\tilde\F_n$ or $Q\subset Q'\in \F_{n-1}$ which is a contradiction. Thus we have proven $\ref{kol}$.

In order to see \ref{nel}, suppose that there exists $Q \in \partial \D_n$ such that $Q \notin \F_n$. Then $Q \in \F_{n,i} \subset \tilde\F_i$ for some $i < n$. By Property \ref{neljas}, for all the $Q' \in \tilde \F$ with $Q' \cap Q \ne \emptyset$ we have $Q' \in \tilde\F_j$ for $j \le i+1 \le n$. Thus, $Q' \subset D_n$ and $Q \notin \partial \D_n$ giving a contradiction.

If property \ref{viis} fails for some $Q \in \partial \D_n$, then for every $Q' \in \tilde \F$  with $Q' \cap Q \ne \emptyset$ we have $Q' \in \tilde\F_i$ for some $i \le n$. Thus, again $Q' \subset D_n$ and $Q \notin \partial \D_n$ giving a contradiction.

Finally, we prove property \ref{kuus}. Since $C_n$ is open it suffices to prove that every Jordan curve is loop homotopic to a constant loop. Suppose this is not the case. Then there exists a Jordan curve $\gamma$ that is not homotopic to a constant loop, and a point $x\in \Omega\setminus C_n$ that lies inside $\gamma$. In particular there exists $Q\in\Q_n$ such that $Q\not\subset D_n$ which lies inside $\gamma$ and for which $Q\cap D_n$ is an edge of a square. Now by similar argument as in \ref{kol} we conclude that $Q\in\D_n$, which is a contradiction.
\end{proof}

The next lemma shows that we can connect the boundary of $D_n$ to the boundary of $\Omega$ with a short curve in the complement of $D_n$.

\begin{lemma}\label{connecting curve}
For each point $x\in \partial D_n$, there exists an injective curve $\gamma\colon [0,1]\to \R^2$ so that $\gamma(0)=x$, $\gamma(1)\in \partial \Omega$, $\gamma(0,1)\subset \Omega\setminus \Int D_n$ and $L(\gamma)\le 2\sqrt{2}l(Q)$.
\end{lemma}

\begin{proof}
Let $Q\in \partial \D_n$ be such that $x\in Q\cap \partial D_n$. By Lemma \ref{properties} \ref{viis} we have that there exists a square $Q'\in \Q_n$ touching $Q$ at $x$ so that $Q'\notin \tilde\F_n$ and $Q'\not\subset \tilde Q $ for every $\tilde Q\in \tilde F_j$, $j<n$. Thus, there exists a neighbouring square $Q''\in \Q_n$ of $Q'$ and a point $y\in \partial \Omega\cap Q''$. Let $\gamma^1$ be a curve corresponding to a line segment connecting $x$ to  a point $z\in Q'\cap Q''$ and let $\gamma^2$ be a curve corresponding to a line segment connecting $z$ to $y$. Moreover, let $t_0\coloneqq \inf\{t: \gamma^2_t\in \partial \Omega\}$. Since $\partial \Omega$ is closed, we have that $\gamma^2_{t_0}\in\partial \Omega$. Define a curve $\gamma\coloneqq \gamma^1*\gamma^2\lvert_{[0,t_0]}$.
For $\gamma$ we have that $\gamma(0,1)\subset (\Omega\setminus \Int D_n)\cap(Q'\cup Q'')$, $\gamma(0)=x$, $\gamma(1)\in\partial \Omega$ and $L(\gamma)\le d(Q')+d(Q'')=2\sqrt{2}l(Q)$.
\end{proof}

Observe that by Lemma \ref{properties} \ref{nel} we have $\partial D_n=\bigcup_{Q\in\partial \D_n}(Q\cap \partial D_n)$. Thus, by Lemma \ref{properties} \ref{kol} we have that $\partial D_n$ is locally homeomorphic to the real line. 
Since by Lemma \ref{properties} \ref{kuus} $C_n$ is simply connected, 
we have that $\partial D_n=\partial C_n$ is connected.
Hence, $\partial D_n$ is a Jordan curve. Thus, we may write \begin{equation}\label{boundary curve}\partial D_n=\bigcup_{i=1}^{L_n}I_i,\end{equation} where $I_i=[y_i,y_{i+1}]$ is an edge of a square in $\F_n$ with vertices $y_i$ and $y_{i+1}$, and $y_1=y_{L_n+1}$.

For the rest of the paper we fix a constant $M>(4\sqrt{2}+2)$. However, the following lemma is true for any $M>0$ and with $C$ depending on $M$.

\begin{lemma}\label{xy}
There exists $C\in\N$ so that for any $n \in \mathbb N$ and $x,y\in \partial D_n$ with $d_{\partial D_n}(x,y)\ge 2^{-n}C$, and for any $\gamma$ in $\Omega\setminus \Int D_n$ connecting $x$ to $y$ we have that $\gamma\cap \left(\Omega\setminus B(x,M2^{-n})\right)\neq \emptyset$. In particular, $L(\gamma)\ge M2^{-n}$.
\end{lemma}
\begin{proof}

By taking a slightly larger $C$, namely $C+2$, we may assume that $x=y_i$ and $y=y_j$ for some $i$ and $j$, where $y_i,y_j$ are two endpoints of intervals from the collection $\{I_i\}$ forming the boundary as noted above.
Moreover, by symmetry we may assume that $i<j$ and $j-i\le n+1-j$.  
Since each $I_i$ is a side for two squares in $\Q_n$, by taking $C$ large enough, we obtain
\[
\mathcal{H}^2(B(x,2(M+1)2^{-n})) = \pi (2(M+1)2^{-n})^2 < \frac12 C (2^{-n})^2 \le \mathcal{H}^2(\bigcup Q),
\]
where the union is taken over all $Q\in \Q_n$ having $I_m$ as one of it sides for some $i<m\le j-1$.
Therefore, one of the intervals $I_{m_1}$, for $i<m_1\le j-1$, has to intersect with the complement of the ball $B(x,2M2^{-n})$. Let $Q_1'\in\partial \D_n$ be the boundary square corresponding to that interval and let $q_1\in I_{m_1}\setminus B(x,2M2^{-n})$. By symmetry, there also exists $Q_2'\in\partial \D_n$ whose side is some $I_{m_2}$ with ${m_2} \notin \{i+1,i+1,\dots, j-1\}$ such that there is $q_2 \in I_{m_2} \setminus B(x,2M2^{-n})$.

Suppose now that there exists a curve $\gamma$ in $\Omega\setminus \Int D_n$ joining $x$ to $y$ with $\gamma\subset B(x,M2^{-n})$. We may assume that $\gamma$ is injective, and by compactness that $\gamma(t)\in \Omega\setminus D_n$ for every $t\in(0,1)$. Then, for $i=1,2$ we have that $B(Q_i',2\sqrt{2}l(Q_i'))\subset B(q,M2^{-n})$ and hence $B(Q_i',2\sqrt{2}l(Q'))\cap \gamma=\emptyset$. Now by definition of $Q_i'$ there is a neighbouring square $Q_i''\in \Q_n$ of $Q_i'$ which is not a subset of $D_n$, see Figure \ref{fig:crosscut}.
We claim that either $Q_1''$ or $Q_2''$ lies inside the Jordan curve $\gamma'$ obtained by concatenating the curve $\gamma$ and the part of the boundary, denoted by $\gamma''$, obtained from the intervals $\{I_h\}_{h=i}^{j-1}$, or by concatenating $\gamma$ and $\partial D_n \setminus \gamma''$. 

This can be seen in the following way. Consider $\Omega\xrightarrow{h}\R^2\hookrightarrow \S^2$, where $h$ is a homeomorphism and the inclusion $\R^2\hookrightarrow \S^2$ is the inverse of the stereographic projection. Under this composite map $\S^2\setminus D_n$ is a simply connected domain. Hence, by Lemma \ref{lma:domaincut} $(\S^2\setminus D_n)\setminus \gamma$ has exactly two components whose boundaries are the two connected components of $\partial D_n \setminus \gamma$ together with $\gamma$. Thus, $(\Omega\setminus  D_n)\setminus\gamma=(\S^2\setminus D_n)\setminus \gamma$ has exactly two components. Since $\partial Q_1'' \cap \partial D_n$ and $\partial Q_2'' \cap \partial D_n$ are in two different connected components of $\partial D_n \setminus \gamma$, we conclude that $Q_1''$ and $Q_2''$ are in different components of $(\Omega\setminus D_n)\setminus \gamma$. We denote the $Q_i''$ that lies inside the Jordan curve by $Q''$.
\begin{figure}
\includegraphics[width=0.5\textwidth]{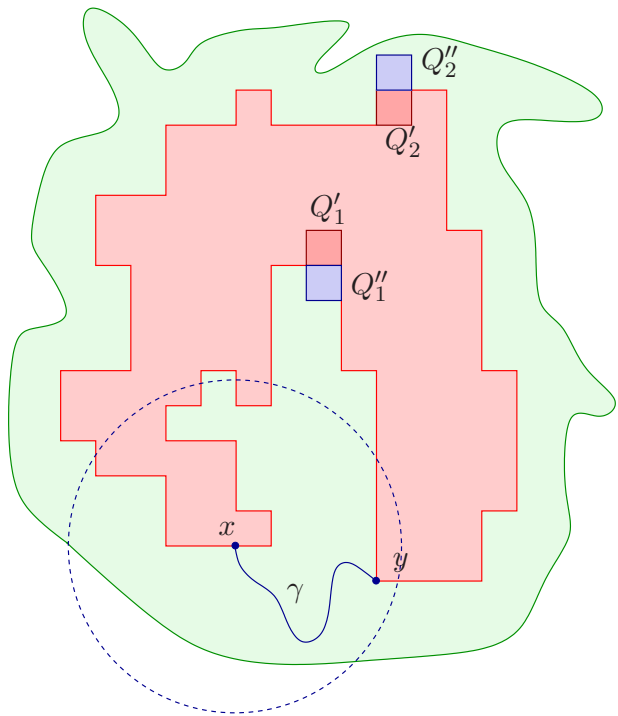}
     \caption{In the proof of Lemma \ref{xy} we assume towards a contradiction that $x$ and $y$ can be connected by a short curve $\gamma$ in $\Omega \setminus D_n$. This will imply that one more square in $\Q_n$ (here $Q_1''$) will be a subset of $D_n$.}
   \label{fig:crosscut}
 \end{figure}

Since $Q''\subset B(Q',\sqrt{2}l(Q'))$, we have that every neighbouring square of $Q''$ either lies inside $\gamma'$ or is an element of $\partial\D_n$. In particular, by the simply connectedness of $\Omega$ they all are subsets of $\Omega$. Hence, $Q''\subset D_n$ which is a contradiction. Thus, we have proven that $\gamma\cap\left(\Omega\setminus B(x,M2^{-n})\right)\neq\emptyset$.
\end{proof}

Let us now partition $\Omega\setminus D_n$ in the following way. 
Recall \eqref{boundary curve}.
Notice that for large enough $n$ we have that $L_n\ge2C$. Define $x_1\coloneqq y_1$ and then $x_m\coloneqq y_{(m-1)C}$ until $L_n+1-(m-1)C<2C$. 
Notice that for every $i \ne j$ we have $d_{\partial D_n}(x_i,x_j)\ge 2^{-n}C$.
We now partition the set $\Omega\setminus D_n$ up to Lebesgue measure zero into connected sets $\{\tilde H_j\}_{j=1}^m$ where $\tilde H_j$ is the open set bounded by $\gamma_j$, $\gamma_{j+1}$ given by Lemma \ref{connecting curve} for points $x_j$ and $x_{j+1}$, and $J_j\coloneqq\bigcup_{i=Cj}^{C(j+1)}I_i$ (with interior in $\Omega\setminus D_n$). This partition is well defined by Lemma \ref{lma:domaincut}. Notice that since $L(\gamma_i)\le M$ for all $i$, we have that $\gamma_i\cap\gamma_j=\emptyset$ for all $i \ne j$. Let us define $H_j$ as the connected component containing $\tilde H_j$ of the set $\Omega\cap\left(\tilde H_j\cup B_{\R^2}(\gamma_j\cup\gamma_{j+1}\cup J_j, \delta)\right)$, where $\delta = 2^{-n-3}$. See Figure \ref{fig:decomp} for an illustration of the decomposition. 
Although the decomposition depends on $n$, for simplicity we do not display the dependence in the notation.
A crucial property of our decomposition is the following lemma.

\begin{figure}
\includegraphics[width=0.7\textwidth]{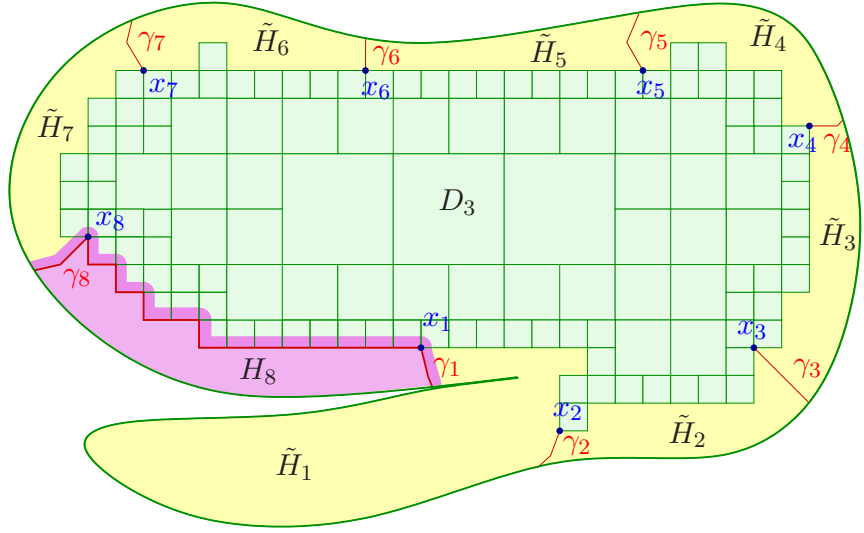}
     \caption{Here the domain $\Omega$ is decomposed into the core part $D_3$ and eight boundary parts $\tilde H_i$. A neighbourhood $H_8$ of $\tilde H_8$ is also illustrated.}
   \label{fig:decomp}
 \end{figure}

\begin{lemma}
We have $H_j\cap H_i \neq \emptyset$, if and only if $\lvert i-j\rvert\le 1$ in a cyclic manner.
\end{lemma}

\begin{proof}
Trivially $\gamma_{i+1} \in H_i \cap H_{i+1}$. Thus, we only need to show that 
$H_j\cap H_i \neq \emptyset$ implies $\lvert i-j\rvert\le 1$. We may assume that $i \ne j$.
Let $x\in H_i\cap H_j$.

Suppose first that $x\in \tilde H_i$. Then, by (path) connectedness of $H_j$ there exists a path $\gamma$ in $H_j$ from $x$ to $\tilde H_j$. Let \begin{align}t_0\coloneqq \inf\{t\in[0,1]:\gamma(t)\notin \tilde H_i\} .\end{align}
Then, $\gamma(t_0)\notin \tilde H_i$ but $\gamma(t)\in H_i\cap H_j$. Thus it suffices to consider the case when $x\notin \tilde H_i\cup \tilde H_j$. 

Suppose now that $x\in D_n$. Since $\delta<\frac{2^{-n}}{2}$, we have that $x\in Q$ for some $Q\in\partial \D_n$. Then, there are neighbouring squares $Q_i, Q_j \in \Q_n$ of $Q$ for which $Q_i\cap \tilde H_i\neq\emptyset$ and $Q_j\cap\tilde H_j\neq\emptyset$. Since $\delta$ is small, we may choose the $Q_i,Q_j$ so that $Q_i\cap Q_j\neq \emptyset$. If $Q_i=Q_j$ or if $Q_i$ and $Q_j$ have a common edge, then there is a curve $\gamma'$ in $Q_i\cup Q_j$ from $\tilde H_i$ to $\tilde H_j$ with $L(\gamma')<2\delta$. If $Q_i\cap Q_j$ is a singleton, then by Lemma \ref{properties} \ref{kol} the neighbouring square $Q'\neq Q$ of both $Q_i$ and $Q_j$ lies in $\Omega\setminus \Int D_n$. Indeed, if this were not the case, then $Q',Q \in \F_n$ and $Q' \cap Q$ is a singleton, implying that $Q_i \in \D_n$ or $Q_j \in \D_n$. 
Thus, there exists a curve $\gamma'$ in $\Omega\setminus D_n$ joining $Q_i$ and $Q_j$ with $L(\gamma')< 4\delta$. 

Now, we have $Q_i \cap J_i \ne \emptyset$ or $Q_i \cap (\gamma_i \cup \gamma_{i+1}) \ne \emptyset$. Notice that $\gamma_i\cap J_i \ne \emptyset \ne \gamma_{i+1}\cap J_i$. By Lemma \ref{connecting curve} we have $\max(l(\gamma_i), l(\gamma_{i+1})) \le 2\sqrt{2}\cdot 2^{-n}$. Combining these observations with the analogous ones for $Q_j$, we have that
 $J_i$ and $J_j$ can be connected by a curve in $\Omega\setminus D_k$ with length less than $4\delta+4\sqrt{2}\cdot2^{-n}<2^{-n}M$.
Hence, we have by Lemma \ref{xy} that $\dist_{\partial D_n}(J_i,J_j)\le C$. Thus, $\lvert{i-j}\rvert\le 1$ in cyclical manner.

We are left with the case where $x\in\Omega\setminus (D_n\cup\tilde H_i\cup \tilde H_j)$. By definition we have that $B(D_n,2\delta)\subset \Omega$. Thus, if $\dist(x,J_i)<\delta$, we may join $x$ to $J_i$ by a curve in $\Omega\setminus \Int D_n$ with length less than $\delta$. If $\dist(x,J_i)\ge\delta$, then $x\in B(\gamma_m,\delta)$, where $m\in\{i,i+1\}$. By path connectedness of $H_i$ there is a curve $\gamma$ in $H_i$ connecting $x$ to $\gamma_i\cup\gamma_{i+1}\cup J_i$. We want to prove that $x$ can be joined to $\gamma_m$ in $\delta$-neighbourhood of $\gamma_m$.  If (a subcurve of) $\gamma$ is not such a curve, then we may define \begin{align}t_0\coloneqq \inf \{t\in[0,1]:\gamma(t)\in B(D_n,\delta)  \}.\end{align}
Then, $\gamma\lvert_{[0,t_0]}\subset B(\gamma_m,\delta)$. Therefore, there exists a point $y\in \gamma_m$ with $d(\gamma(t_0),y)<\delta$. In particular, the line segment $[\gamma(t_0),y]$ lies in $(\Omega\setminus D_n)\cap B(\gamma_m,\delta)$ and thus we have proven that there exists a curve $\gamma'$ in $(\Omega \setminus D_n)\cap B(\gamma_m,\delta)$ connecting $x$ to $\gamma_m$. By the definition of $\gamma_m$ we have that $\gamma'\subset B(\gamma_m(0), 2\sqrt{2}\cdot 2^{-n}+\delta)$. By the same argument for $j$ we conclude that $J_i$ and $J_j$ can actually be connected by a curve $\gamma$ in $\left(\Omega\setminus \Int D_n\right)\cap B(\gamma(0),4\sqrt{2}\cdot 2^{-n}+2\delta)$. Hence, by Lemma \ref{xy} $\dist_{\partial D_n}(J_i,J_j)< C$, and thus $\lvert i-j \rvert\le 1$ in cyclical manner.
\end{proof}

\section{Approximation}\label{sec:approx}

In this section we finish the proof of Theorem \ref{main_thm} by making a partition of unity using the decomposition of $\Omega$ constructed in Section \ref{sec:decomp} and by approximating a given function by polynomials in this decomposition. Recall that our aim is to show that for any $u \in L^{k,p}(\Omega)$ and $\epsilon > 0$ there exists a function $u_\epsilon \in W^{k,\infty}(\Omega)\cap C^{\infty}(\Omega)$ with
$\|\nabla^k u-\nabla^k u_{\epsilon}\|_{L^p(\Omega)}\lesssim \epsilon.$
By noting that $L^{k,p}(\Omega)\cap C^{\infty}(\Omega)$ is dense in $L^{k,p}(\Omega)$ we may assume that function $u \in L^{k,p}(\Omega)\cap C^{\infty}(\Omega)$. From now on, let $u$ and $\epsilon > 0$ be fixed.

Using the notation from Section \ref{sec:decomp},
we write the domain $\Omega$ as the union of the core part $D_{n}$ and the boundary regions $\{H_i\}_{i=1}^l$.
For each $\tilde{H}_i$ we let $\mathcal{I}_i$ be the collection of squares $Q$ in $\partial \mathcal{D}_{n}$ such that $Q\cap \tilde{H}_i\neq \emptyset$, which are bounded in number independently of $n$. We need to decide what polynomial to attach to each set $H_i$.
For this purpose, for each $1\leq i\leq l$ we assign a square $Q_i \in\mathcal{I}_i$. We call $Q_i$ the associated square of ${H}_i$.

Given $Q\in \mathcal{I}_i$ we set $\mathcal{P}_Q:=\bigcup_{j=i-1}^{i+1} \{Q'\in\mathcal{I}_j\}$, which is a collection of squares from a suitable neighbourhood of $Q$.

Recall the approximating polynomials $P_Q$ introduced in Section \ref{sec:approxpoly}.
We abbreviate $P_i = P_{Q_i}$ for the associated squares $Q_i$.

We make a smooth partition of unity by using a Euclidean mollification. (Compare to \cite{KZ} where the inner distance in $\Omega$ was used for the mollification.) Let $\rho_r$ denote a standard Euclidean mollifier supported in $B(0,r)$.
We start with a collection of functions $\{\tilde\psi_i\}_{i=0}^l$, where 
$\tilde\psi_0 = \chi_{D_n} \ast \rho_{2^{-n-5}}$ and 
$\tilde\psi_i = \left(\chi_{\tilde H_i} \ast \rho_{2^{-n-5}}\right)|_{H_i}$ for $i \ge 1$.
Using this we obtain a partition of unity $\{\psi_i\}_{i=0}^l$ by setting
$\psi_i = \tilde\psi_i/\sum_{j=0}^l \tilde\psi_j$.

Now the partition of unity $\{\psi_i\}_{i=0}^l$ satisfies the following. 
\begin{enumerate}
\item The function $\psi_0$ is supported in $B(D_{n},\frac{2^{-n}}{10})$.
\item For $i \ge 1$ the function $\psi_i$ is supported in $H_i$.
\item For all $i$, $0\leq \psi_i\leq 1$.
\item $\sum \psi_i \equiv 1$ on $\Omega$.
\item For all $i$, $|\nabla^{\alpha}\psi_i|\leq C_{\alpha}2^{-n|\alpha|}$ for all multi-indeces $\alpha$.
\end{enumerate}
 

 We will fix $n$ later such that for the function $u_\epsilon$ defined as
 $$u_{\epsilon}(x):=u(x)\psi_0(x)+\sum_{i=1}^{l}\psi_i(x)P_i(x)$$ for $x\in\Omega$, we have $$\|\nabla^k u-\nabla^k u_{\epsilon}\|_{L^p(\Omega)}<C\epsilon.$$ 
Note that $u_\epsilon = u$ on $D_{n-1}$; indeed $D_{n-1}\cap \psi_i=\emptyset$ for $i\geq 1$, see Lemma \ref{properties} \ref{nel}.

First of all, we consider only $n$ large enough so that 

 \begin{equation}\label{est_(-1)}
 \|\nabla^k u\|_{L^p(\Omega\backslash D_{n-1})}\leq \epsilon.
 \end{equation}
 
Now, we need to show that $n$ can actually be chosen large enough so that also
 $$\|\nabla^k u_{\epsilon}\|_{L^p(\Omega\backslash D_{n-1})}\leq C\epsilon.$$
 So, we compute for $Q\in \mathcal{I}_i$ and $|\alpha|=k$
 \begin{equation}\label{est_0}
 \begin{split}
  \|\nabla^{\alpha} u_{\epsilon}\|_{L^p(Q)} &\leq \sum_{\beta\leq\alpha}\left(\int_Q|\nabla^{\beta} u-\nabla^{\beta}P_i(x)|^p|\nabla^{\alpha-\beta}\psi_0(x)|^p\,\d x\right)^{1/p}
  \\&+ \sum_{\beta\leq\alpha}\sum_{j}\left(\int_Q|\nabla^{\beta} P_j(x)-\nabla^{\beta}P_i(x)|^p|\nabla^{\alpha-\beta}\psi_j(x)|^p\,\d x\right)^{1/p}
  \\&=: A_1+A_2,
 \end{split}
 \end{equation}

 where $A_1$ and $A_2$ are the first and second terms on the right hand side of the inequality and we used that for $\beta<\alpha$, $\sum_j \nabla^{\alpha-\beta} \psi_j =0$ and order of $P_i$ is at most $k-1$. We first estimate $A_1$ as 
 \begin{equation}\label{est1}
  \begin{split}
   A_1 &\lesssim \sum_{\beta\leq\alpha}2^{n(|\alpha|-|\beta|)}\|\nabla^{\beta}u-\nabla^{\beta}P_i\|_{L^p(Q)}
   \\&\lesssim \sum_{\beta\leq\alpha}2^{n(|\alpha|-|\beta|)}(\|\nabla^{\beta}P_i-\nabla^{\beta}P_Q\|_{L^p(Q)}+\|\nabla^{\beta}u-\nabla^{\beta}P_Q\|_{L^p(Q)})
   \\&\lesssim \sum_{\beta\leq\alpha}2^{n(|\alpha|-|\beta|)}2^{n(|\beta|-k)}\|\nabla^{k} u\|_{L^p(\cup\tilde{Q})}
   \\&\lesssim\|\nabla^{k} u\|_{L^p(\cup\tilde{Q})},
  \end{split}
 \end{equation}
 where in the third inequality we used that $Q_i$ (associated square of $\tilde{H}_i$) and $Q$ may be joined by a chain of bounded number of squares from $\mathcal{I}_i$ by our construction, and therefore we may apply Lemma \ref{chaining}. Similarly we estimate $A_2$ as 
 \begin{equation}\label{est_2}
  \begin{split}
   A_2 &\lesssim\sum_{\beta\leq\alpha}\sum_{j=i-1}^{i+1}\left(\int_Q|\nabla^{\beta} P_j(x)-\nabla^{\beta}P_i(x)|^p|\nabla^{\alpha-\beta}\psi_j(x)|^p\,\d x\right)^{1/p}
   \\&\lesssim\sum_{\beta\leq\alpha}2^{n(|\alpha|-|\beta|)}\sum_{j=i-1}^{i+1}(\|\nabla^{\beta}P_j-\nabla^{\beta}P_Q\|_{L^p(Q)}+\|\nabla^{\beta}P_i-\nabla^{\beta}P_Q\|_{L^p(Q)})
   \\&\lesssim \sum_{\beta\leq\alpha}2^{n(|\alpha|-|\beta|)}2^{n(|\beta|-k)}\|\nabla^{k} u\|_{L^p(\cup\tilde{Q})}
   \\&\lesssim\|\nabla^{k} u\|_{L^p(\cup\tilde{Q})},
  \end{split}
 \end{equation}
 where again in the second inequality we used that if $\psi_j(x)\neq 0$ for $x\in Q\in\mathcal{I}_i$ then by our construction $Q_j$ and $Q$ can be joined by a chain of bounded number of squares as $j$ is either $i-1,i$ or $i+1$ (cyclically); and therefore we can apply Lemma \ref{chaining}.
 
For $Q\in \tilde{\mathcal{F}}\setminus \mathcal{D}_{n}$ 
such that $Q\cap \spt(\psi_0)\neq\emptyset$, we assign to $Q$ a square $Q'\in\mathcal{I}_i$, such that $Q\cap Q'\neq\emptyset$. Note that such a square $Q'$ exists by our construction. Then $Q$ and $Q'$ can be joined by a chain of bounded (by an absolute constant) number of squares from $\mathcal{D}_{n+1}$. We choose such a chain for $Q$ and denote it by $\mathcal{B}_Q$. We also set $$\mathcal{J}_{n}:=\{Q\in \tilde{\mathcal{F}}\setminus \mathcal{D}_{n}:Q\cap \spt(\psi_0)\neq\emptyset\}.$$ We estimate using Lemma \ref{chaining} exactly as above (see \eqref{est_0}) to obtain for $|\alpha|=k$

\begin{equation}\label{est_3}
\begin{split}
  \|\nabla^{\alpha} u_{\epsilon}\|_{L^p(Q)} &\leq \sum_{\beta\leq\alpha}\left(\int_Q|\nabla^{\beta} u-\nabla^{\beta}P_Q(x)|^p|\nabla^{\alpha-\beta}\psi_0(x)|^p\,\d x\right)^{1/p}
  \\&+ \sum_{\beta\leq\alpha}\sum_{j}\left(\int_Q|\nabla^{\beta} P_j(x)-\nabla^{\beta}P_Q(x)|^p|\nabla^{\alpha-\beta}\psi_j(x)|^p\,\d x\right)^{1/p}
  \\&=: B_1+B_2.
 \end{split}
\end{equation}
Again, we estimate separately,
\begin{equation}\label{est_4}
  \begin{split}
   B_1 &\lesssim \sum_{\beta\leq\alpha}2^{n(|\alpha|-|\beta|)}\|\nabla^{\beta}u-\nabla^{\beta}P_Q\|_{L^p(Q)}
   \\&\lesssim\|\nabla^{k} u\|_{L^p(Q)}
  \end{split}
 \end{equation}
and
\begin{equation}\label{est_5}
  \begin{split}
   B_2 &\lesssim\sum_{\beta\leq\alpha}\sum_{j=i-1}^{i+1}\left(\int_Q|\nabla^{\beta} P_j(x)-\nabla^{\beta}P_Q(x)|^p|\nabla^{\alpha-\beta}\psi_j(x)|^p\,\d x\right)^{1/p}
   \\&\lesssim\sum_{\beta\leq\alpha}2^{n(|\alpha|-|\beta|)}\sum_{j=i-1}^{i+1}(\|\nabla^{\beta}P_j-\nabla^{\beta}P_{Q'}\|_{L^p(Q)}+\|\nabla^{\beta}P_{Q'}-\nabla^{\beta}P_Q\|_{L^p(Q)})
   \\&\lesssim \sum_{\beta\leq\alpha}2^{n(|\alpha|-|\beta|)}2^{n(|\beta|-k)}(\|\nabla^{k} u\|_{L^p(\underset{Q''\in\mathcal{P}_{Q'}}\bigcup Q'')}+\|\nabla^{k} u\|_{L^p(\underset{Q''\in \mathcal{B}_{Q'}}\bigcup Q'')}
   \\&\lesssim \|\nabla^{k} u\|_{L^p(\underset{Q''\in\mathcal{P}_{Q'}\bigcup \mathcal{B}_{Q'}}\bigcup Q'')}.
  \end{split}
 \end{equation}
 Next we note that $\nabla^k u_{\epsilon}\equiv 0$ in $\tilde{H}_i\backslash \cup_{j\neq i} \spt(\psi_j)$ and we compute for $|\alpha|=k$
 
  \begin{equation}\label{est_6}
 \begin{split}
  \|\nabla^{\alpha} u_{\epsilon}\|_{L^p(H_i)} &\leq \sum_{Q\in \tilde{Q_i}}\|\nabla^{\alpha} u_{\epsilon}\|_{L^p(Q)} + \sum_{Q\in \mathcal{J}_{n},Q\cap H_i\neq\emptyset} \|\nabla^{\alpha} u_{\epsilon}\|_{L^p(Q)}
  \\&+ \sum_{j=i-1}^{i+1}\|\nabla^{\alpha} u_{\epsilon}\|_{L^p((\spt(\psi_j)\cap \spt(\psi_i))\setminus \underset{Q''\in\partial\mathcal{D}_{n}\bigcap \mathcal{J}_{n}}\bigcup Q'')}
 \end{split}
 \end{equation}
 The terms in the first and second summands have been estimated earlier. Denoting $H'_i:=(\cup_{j=i-1}^{i+1}\spt(\psi_j)\cap \spt(\psi_i))\setminus \underset{Q''\in\partial\mathcal{D}_{n}\bigcap \mathcal{J}_{n}}\bigcup Q''$, we estimate now the third one;
 
 \begin{equation}\label{est_7}
  \begin{split}
   \|\nabla^{\alpha} u_{\epsilon}\|_{L^p(H'_i)} 
   &\lesssim \sum_{\beta\leq\alpha}
    2^{n(|\alpha|-|\beta|)}\sum_{j=i-1}^{i+1}\|\nabla^{\beta}P_j-\nabla^{\beta}P_i\|_{L^p(H'_i)}
   \\&\lesssim \sum_{\beta\leq\alpha}2^{n(|\alpha|-|\beta|)}(\|\nabla^{\beta}P_j-\nabla^{\beta}P_i\|_{L^p(Q_i)}
   \\&\lesssim \sum_{\beta\leq\alpha}2^{n(|\alpha|-|\beta|)}2^{n(|\beta|-k)}\|\nabla^{k} u\|_{L^p(\underset{Q'\in\mathcal{P}_{Q_i}}\bigcup Q')}
   \\&\lesssim\|\nabla^{k} u\|_{L^p(\underset{Q'\in\mathcal{P}_{Q_i}}\bigcup Q')},
  \end{split}
 \end{equation}
where we used the facts that for $\beta<\alpha$, $\nabla^{\alpha-\beta}\sum_j(\psi_j)=0$ and $\psi_0\equiv 0$ in $H'_i$ in the first inequality, Lemma \ref{norm_equivalence} in the second inequality since $H'_i\subset CQ_i$ for some absolute constant $C$ coming from Lemma \ref{connecting curve} and in the third inequality we used Lemma \ref{chaining}.

 \begin{rem}\label{overlap}
Note that for each $Q\in\mathcal{I}_i$ we have $\mathcal{P}_{Q}=\mathcal{P}_{Q_i}$ where $Q_i$ is the associated square of $\tilde{H}_i$.
We note that any $Q'\in\partial\mathcal{D}_n$ occurs in at most three distinct collections $\mathcal{P}_{Q_i}$. Moreover any $Q\in\mathcal{D}_{n+1}$ appears in only a bounded number of the collections $\mathcal{B}_{Q''}$, where $Q''\in\mathcal{J}_{n}$. In particular, any $Q'\in\partial\mathcal{D}_n$ appears in only a bounded number of the collections $\mathcal{B}_{Q''}$, where $Q''\in\mathcal{J}_{n}$. The bounds are provided by absolute constants coming from volume comparison. 
\end{rem}

Now it follows from equations \eqref{est1}, \eqref{est_2}, \eqref{est_3}, \eqref{est_6} and \eqref{est_7} that

\begin{equation}\label{est_8}
\begin{split}
\|\nabla^{\alpha} u_{\epsilon}\|_{L^p(\Omega\backslash C_{n})}&\lesssim \sum_i\|\nabla^ k u\|_{L^p(H_i)}+\|\nabla^k u\|_{L^p(\underset{Q\in\partial \mathcal{D}_{n}}\bigcup Q)}
\\& \lesssim \|\nabla^k u\|_{L^p(\underset{Q\in\partial \mathcal{D}_{n}}\bigcup Q)} + \|\nabla^k u\|_{L^p(\underset{Q\in\mathcal{J}_{n}}\bigcup Q)} + \|\nabla^k u\|_{L^p(\underset{\substack{Q\in\mathcal{J}_{n}\\ Q'\in \mathcal{B}_Q}}\bigcup Q')}
\end{split}
\end{equation} 
when $|\alpha|=k$. By Remark \ref{overlap} we may choose $n$ such that
\[\|\nabla^k u\|_{L^p(\underset{Q\in\partial \mathcal{D}_{n}}\bigcup Q)} + \|\nabla^k u\|_{L^p(\underset{Q\in\mathcal{J}_{n}}\bigcup Q)} + \|\nabla^k u\|_{L^p(\underset{\substack{Q\in\mathcal{J}_{n}\\ Q'\in \mathcal{B}_Q}}\bigcup Q')}<\epsilon.
\]
Then, the claim follows from \eqref{est_(-1)} and \eqref{est_8}.

\begin{rem}\label{k=1}
We note that when $k=1$ we may take the function to be smooth as well as bounded for showing the density of $W^{1,\infty}(\Omega)$ in $W^{1,p}(\Omega)$. This is because truncations approximate the functions in $W^{1,p}(\Omega)$. This allows us to boundealso approximate the $L^p$ norm of $u$. Indeed let $u\in W^{1,p}(\Omega)\cap C^{\infty}(\Omega)\cap L^{\infty}(\Omega)$ such that $\|u\|_{L^{\infty}}\leq M$. Decompose the domain as in the above construction; then choose $n$ large enough such that $\|u\|_{W^{1,p}(\Omega\setminus D_{n-1})}\leq \epsilon$ and $M|\Omega\setminus D_{n-1}|<\epsilon$. Then it follows from estimates in the proof that the function $u_{\epsilon}$ defined as above approximates $u$ in $W^{1,p}(\Omega)$ with error given by $\epsilon$. This conclusion is the content of \cite{KZ}.
\end{rem}

Finally, let us show how the smooth approximation in Jordan domains is done.

\begin{proof}[Proof of Corollary 1.3]
The argument we need follows the one used to prove \cite[Corollary 1.2]{KZ}. As in \cite{KZ}, given a bounded Jordan domain we approximate it from outside by a nested sequence of Lipschitz and simply connected domains $G_s$ which are obtained for example by taking the complement of the unbounded connected component of the union Whitney squares larger than $2^{-s}$ from the Whitney decomposition of the complementary Jordan domain of $\Omega$. 

Then, we note that for given $n$, taking $s_n$ large enough, we have that the squares in $\partial \mathcal{D}_n$ are Whitney type sets in $G_{s_n}$, meaning they have diameters comparable to the distance from the boundary of $G_{s_n}$. 

Note that $G_{s_n}\subset B(\Omega,2^{-s_n+5})$ are simply connected. Now the set $G_{s_n}\setminus \bar{C_n}$ (recall that $C_n$ is a suitable connected component of the interior of the union of the Whitney squares of scale less than $2^{-n}$) can be decomposed in the same way as $\Omega\setminus \bar{C_n}$ was decomposed into the sets $\tilde{H}_i$ in Section \ref{sec:decomp}. 

We may then follow the argument used in the proof of Theorem \ref{main_thm} to obtain an approximating sequence of functions $u_n$ in $G_{s_n}$ which are in the space $W^{k,\infty}(G_{s_n})\cap L^{k,p}(G_{s_n})\cap C^{\infty}(G_{s_n})$. 
By multiplying with a smooth cut-off function that is $1$ on $\Omega$ and compactly supported in $G_{s_n}$, we obtain a sequence of global smooth functions having the desired properties.
\end{proof}

\bibliographystyle{amsplain}
\bibliography{SoboReferences}
\end{document}